\documentclass[12pt]{article}

\usepackage[latin1]{inputenc}
\usepackage{epsfig}
\usepackage{graphicx,psfrag}
\usepackage[tight]{subfigure}

\usepackage{amsfonts,amsthm,bbm,amssymb,epsf,amsmath,%
latexsym,amscd,amsfonts,enumerate,amsthm,supertabular,color,url,multirow}
\usepackage[colorlinks,bookmarksopen,bookmarksnumbered,citecolor=blue,urlcolor=blue]{hyperref}

\newtheorem{theorem}{Theorem}[section]
\newtheorem{lemma}[theorem]{Lemma}
\newtheorem{corollary}[theorem]{Corollary}
\newtheorem{proposition}[theorem]{Proposition}
\newtheorem{example}[theorem]{Example}

\newtheorem*{keywords}{Keywords}

	\let\Item\item

\def\address#1{\expandafter\def\expandafter\@aabuffer\expandafter
	{\@aabuffer{\affiliationfont{#1}}\relax\par
	\vspace*{13pt}}}

\numberwithin{equation}{section}

\begin{document}

\date{}

\title{On Krull dimension of modules over group rings of minimax abelian groups}

\author{ Anatolii V. Tushev 
\thanks{The author has received funding through the MSCA4Ukraine project, which is funded by the European Union (ID 1232926)} \\
        Justus Liebig University Giessen\\
        Giessen 35390, Germany;\\
        {\it E-mail: Anatolii.Tushev@math.uni-giessen.de}\\
       }

\maketitle

\begin{abstract}
In the paper we obtained some estimations of Krull dimension of modules over group rings of minimax abelian groups. We also consider relations between the condition of existing of small deviation for normal subgroups and some previously studied chain conditions. 

 \begin{keywords}
minimax groups, Krull dimension, group rings
  \\
\end{keywords}
\end{abstract}

\section{Introduction}\label{section1}

A group $G$ is said to have finite (Prufer) rank if there is a positive integer $m$ such that any finitely generated subgroup of $G$ may be generated by $m$ elements; the smallest $m$ with this property is the rank $r(G)$ of $G$. \par
Let $\Omega $ be a set of subgroups of a group $G$ then $\Omega $ is a partially ordered set (poset for short) with respect to inclusion. The group $G$ meets the minimal (maximal) condition $Min - \Omega $  ($Max - \Omega $) for subgroups from $\Omega $ if $G$ has no infinite descending (ascending) chain of subgroups from $\Omega $. The group $G$ meets the weak minimal (maximal) condition $Min - \infty  - \Omega $ ($Max - \infty  - \Omega $) for subgroups from $\Omega $ if $\left| {{H_i}:{H_{i + 1}}} \right| < \infty $ ($\left| {{H_{i + 1}}:{H_i}} \right| < \infty $) almost for all $i$ for any infinite descending (ascending) chain $\{ {H_i}\} $ of subgroups from $\Omega $. \par
By a Chernikov result, a soluble group $G$ meets $Min$ if and only if $G$ is a finite extension of a direct product of finitely many of quasicyclic groups; now, such groups are called Chernikov. A soluble group $G$ meets $Max$ if and only if it is polycyclic, that is $G$ has a finite subnormal series with cyclic factors. \par
A group $G$ is said to be minimax if it has a finite subnormal series each of whose factor meets either $Min$ or $Max.$ It follows from the mentioned above description of soluble groups with $Min$ and $Max$ that a soluble group $G$ is minimax if and only if it has a finite subnormal series each of whose factor is either cyclic or quasicyclic. A locally soluble group $G$ meets $Min - \infty $ or $Max - \infty $ if and only if $G$ is a soluble minimax group. \par
Let $\Omega $ be a poset. For any elements $a,b \in \Omega $ we put $a/b = \{ x \in \Omega |a \le x \le b\} $. The deviation of $\Omega $ can be defined as the following: \par
1. if $\Omega $ is trivial then $dev\Omega  =  - 1$ \par
2. if $\Omega $ meets the minimal condition then $dev\Omega  = 0$ \par
3. for a general ordinal $\beta $, $dev\Omega  = \beta $ if : \par
( i ) $dev\Omega  \ne \alpha  < \beta $,\par
( ii ) in any descending chain $\{ {a_i}\} $ of elements of $\Omega $ all but finitely many factors ${a_i}/{a_{i + 1}}$ have deviation less than $\beta $. \par
We should note that a poset has deviation if and only if it contains no sub-poset order isomorphic to the poset  of rational numbers in their usual ordering (see \cite[Introduction]{DaMa2024}).\par
Let $R$ be a ring. If $\Omega $ is the set of all submodules of an $R - $module $M$ then $dev\Omega $ is called Krull dimension of $M$ which is denoted by ${K_R}(M)$. The Krull dimension of a right module ${R_R}$ is called the Krull dimension of the ring $R$ and denoted by $K(R)$. In \cite{Tush2003} it was proposed to consider groups with deviation in posets of subgroups. If $G$ is a group and $\Omega $ is a poset of subgroups of $G$ then $de{v_\Omega }G$ denotes the deviation of $\Omega $. If $\Omega $ is the poset of all subgroups of $G$ then we put $de{v_\Omega }G = devG$. The arguments of the proof of \cite[Proposition 6.1.8 ]{McRo1988} show that if a group $G$ meets $Max - \infty  - \Omega $ then $de{v_\Omega }G$ exists. Thus, the condition of existence of $de{v_\Omega }G$ generalizes both $Min - \infty  - \Omega $ and $Max - \infty  - \Omega $. However, 
\cite[Lemma 4.4]{Tush2003} shows that a soluble group has $devG$ if and only if $G$ is minimax. 
The condition of existence of $de{v_\Omega }G$ was considered for various posets $\Omega $ of 
subgroups (see \cite{GKR1, GKR2, KuSm2009}). \par
If $\Omega $ is the poset of all normal subgroups of a group $G$ then we call $de{v_\Omega }G$ by Krull dimension of $G$ and denote it by $K(G)$. Evidently, if a group $G$ meets $Min - \infty  - \Omega $ then $de{v_\Omega }G \le 1$. A group $G$ is said to be $G$-minimax if it has a finite series of normal subgroups each of whose factor satisfies either minimal or maximal condition for $G$-invariant subgroups. It was proved in \cite[Theorem 4.5]{Tush2003} that a metabelian group $G$ has Krull dimension if and only if it is $G$-minimax. \par
According to \cite{DaMa2024} a poset of subgroups of a group has the real chain condition ($RCC$) if it does not contain chains whose order type is the same as the set of real numbers  with their usual ordering. It is easy to note that if a poset $\Omega $ of subgroups of a group $G$ has $de{v_\Omega }G$ then $\Omega $ meets $RCC$. However, unlike other chain conditions, existence of deviation of a poset $\Omega $ of subgroups of a group $G$ gives us a number characteristic $di{v_\Omega }G$ of the group $G$ (in the general case the number $di{v_\Omega }G$ may be transfinite). So, one can study properties of the group $di{v_\Omega }G$ which depend on the value of $di{v_\Omega }G$. \par
In Section 2 we obtained some estimations of Krull dimension of modules over integer group rings of minimax abelian groups (Theoren 2.4). Our estimations are based on techniques developed in \cite{Tush2002, Tush2003}. Then we apply our results on modules over integer group rings for studying metabelian groups with $K(G) \le 1$. If $G$ is a metabelian group and $K(G) = 0$ then $G$ satisfies the condition $Min - N$ of minimality for normal subgroup, such groups were studied in \cite{HaMc1971}. \par
In Section 3 we study properties of metabelian groups with $K(G) \le 1$. In particular, we consider relations between the condition $K(G) \le 1$ and the condition no-${\mathbb{Z}} \wr {\mathbb{Z}}$-sections which was studied in \cite{JaKr2020}. We also consider relations between the condition $K(G) \le 1$ and the weak condition of minimality $Min - \infty  - N$ for normal subgroups which was studied in \cite{ZaKT1985}. 

\section{Some estimations of Krull dimension of modules over minimax abelian groups}\label{section2}

	Let $G$ be a minimax abelian group then  it is not difficult to note that there is a unique minimal characteristic subgroup $P\left( G \right)$ of $G$ such that the quotient group $G/P\left( G \right)$ is free abelian. \par
Let $R$ be a ring then $U(R)$ denotes the group of units of $R$. We define the dimension $dim\, R$ of $R$ 
 to be the supremum of the lengths of all chains of prime ideals in $R$. If $F \ge k$ is a field extension then $tr - deg\left( {F/k} \right)$ denotes the transcendence degree of $F$ over $k$. If $X$ is a subset of $F$ then $k(X)$ denotes a subfield generated by $k$ and $X$ while $k[X]$ denotes a subring generated by $k$ and $X$. \par

	\begin{lemma}\label{Lemma 2.1}
	 Let $R$ be a commutative domain. Then:\par
	(i) if the ring $R$ is Noetherian then $dim\, R = K(R)$;  \par
	(ii) if $S$ is a subring of $R$ such that $R$ is integer over $S$ then $dim\, R = dim\, S$; \par
	(iii) if $R$ is a finitely generated $k$-algebra, where $k$ is a subfield of $R$, then $dim\, R = K(R) = tr-deg({F/k})$, where $F$ is the field of fractions of $R$. 
\end{lemma}
\begin{proof}
	(i) The assertion follows from \cite[Corollary 6.4.8]{McRo1988}.\par
	(ii) The assertion follows from the Cohen-Seidenberg theorem.\par
	(iii) The assertion follows from the Noether normalization lemma and (ii).   
\end{proof}

Let $R$ be a ring and $M$ be an $R$-module which has Krull dimension. The module $M$ is said to be critical if $K(M/{M_1}) < K(M)$ for any nonzero submodule ${M_1} \le M$. If $p$ is a prime integer then 
${\mathbb{ Z}}_p$ denotes a field of order $p$.  If  $G$ is an abelian group then $Sp(G)$ is the set 
of prime integers $p$ such that $G$ has an infinite $p$-section. If $T$ is a torsion group then $\pi(T)$ 
denotes the set of prime divisors of orders of elements of $T$. For a given set $ \pi $ of prime integer 
we say that an additive abelian group $ G $ is $ \pi $-divisible if $ Gp = G $ for any $p\in  \pi $.

\begin{lemma}\label{Lemma 2.2}
 Let $G$ be a minimax abelian group and let $M = a{ \mathbb{ Z} } G$ be a cyclic critical faithful ${ \mathbb{ Z} } G$-module. Then $M \cong R$, where $R$ is a domain and: \par
	(i) $G = T \times H \le U(R)$, where $H$ is a finitely generated free abelian group and $T$ is a Chernikov group; \par
(ii) either $R = {{ \mathbb{ Z} }_p}[G]$ and $p \notin \pi (T)$ or $R = { \mathbb{ Z} } [G]$ the additive group of $R$ is torsion-free $\pi $-divisible, where $\pi  = Sp(G)$;\par
(iii) the module $M$ is Noetherian and $K(M) \le r(G/P(G)) + 1$. 
\end{lemma}

\begin{proof}
By \cite[Lemma 3.1]{Tush2003}, $an{n_{{ \mathbb{ Z} } G}}(a) = P$ is a prime ideal of ${ \mathbb{ Z} } G$ and hence $M = a{ \mathbb{ Z} } G \cong { \mathbb{ Z} } /P = R$, where $R$ is a domain. \par
(i) Let $F$ be the field of fractions of the domain $R$ and ${F^*}$ be the multiplicative group of $F$. Since the module $M$ is faithful, we can assume that $G \le {F^*}$ and the subring $R$ of $F$ is generated by $G$. Since the module $M$ has Krull dimension, the ring $R$ also has Krull dimension. Therefore, by \cite[Theorem 3.4(i)]{Tush2003}, $G = T \times H$, where $H$ is a finitely generated free abelian group and $T$ is a Chernikov group. \par
	(ii) If $char\,F = p > 0$ then  $R = {{ \mathbb{ Z} }_p}[G]$ and $p \notin \pi (T)$. If $char\,F = 0$ then the additive group $R^+$ of $ R= { \mathbb{ Z} }[G]$ is torsion-free and hence, by \cite[Theorem 3.4(i)]{Tush2003}, $R^+$ is $\pi $-divisible, where $\pi  = Sp(G)$. \par
	(iii) It follows from \cite[Theorem 3.4(iii)]{Tush2003} that $M$ is a Noetherian ${ \mathbb{ Z} } G$-module. 
\end{proof}

	Let $\pi $ be a finite set of prime integers, we say that an integer $n$ is a $\pi $-integer if all prime divisors of $n$ are contained in $\pi $. It is easy to note that fractions of ${ \mathbb{ Q} } $ whose denominators are $\pi $-integers form a subring ${{ \mathbb{ Q} }_\pi }$ of ${ \mathbb{ Q} } $. It is not difficult to check that $K({{ \mathbb{ Q} }_\pi }) = dim\, ({{ \mathbb{ Q} }_\pi }) = 1$
	
\begin{proposition}\label{ Proposition 2.3}
 Let $F$ be a field and $G = T \times H$ be a subgroup of the multiplicative group ${F^*}$ of $F$, where $H$ is a finitely generated free abelian group and $T$ is a Chernikov group. Let $R$ be a subring of $F$ generated by the subgroup $G$. Suppose that the ring $R$ has Krull dimension $K(R)$ and $F$ is the field of fractions of the domain $R$. Then:\par
(i) if $char\,F = p > 0$ then $K(R) = tr-deg({{\mathbb{Z}}_p}(H)/{{ \mathbb{Z}}_p})$= 
$tr-deg({{\mathbb{Z}}_p}(G)/{{\mathbb{Z}}_p})$; \par
(ii) if $char\,F = 0$ then $K(R) = tr-deg ({\mathbb{Q}}(H)/{\mathbb{Q}}) + 1 =
 tr-deg({\mathbb{Q}}(G)/{\mathbb{Q}}) + 1$.   
\end{proposition}

\begin{proof}
It follows from Lemmas 2.1(i) and 2.2(iii) that $dim\, R = K(R)$ and hence in the proof we can replace $K(R)$ with $dim\, R$.\par
(i) Since the quotient group $G/H$ is torsion, it is not difficult to show that $R$ is integer over ${{ \mathbb{ Z} }_p}[H]$. Then it follows from Lemma 2.1(ii) that $dim\, R = dim\, {{ \mathbb{ Z} }_p}[H]$. Therefore, as the subgroup $H$ is finitely generated, it follows from Lemma 2.1(iii) that $dim\, (R) = tr-deg({{\mathbb{Z}}_p}(H)/{{ \mathbb{ Z} }_p})$. \par
(ii) Since the quotient group $G/H$ is torsion, it is not difficult to show that $R$ is integer over ${ \mathbb{ Z} } [H]$. Then it follows from Lemma 2.1(ii) that $dim\, R = dim\, { \mathbb{ Z} } [H]$. It follows from the Noether normalization lemma that ${ \mathbb{ Q} } [H] = ({ \mathbb{ Q} } [X])[Y]$, where $X$ is a maximal set of algebraically independent over ${ \mathbb{ Q} } $ elements and $Y$ is a finite set of integer over  ${ \mathbb{ Q} } [X]$ elements. It is not difficult to show that there is a finite set $\pi $ of prime integers such that elements of $Y$ are integer over ${{ \mathbb{ Q} }_\pi }[X]$. Then it follows from Lemma 2.1(ii) that $dim\, {{ \mathbb{ Q} }_\pi }[H] = {{ \mathbb{ Q} }_\pi }[X] = 1 + \left| X \right|$=$tr-deg({\mathbb{Q}}(H)/{ \mathbb{Q}}) + 1$.\par
It is easy to note that if $Q \ne P$ prime ideals of ${\mathbb{Q}_\pi }[H]$ then ${\mathbb{ Z} } [H] \cap Q \ne {\mathbb{ Z} } [H] \cap P$ and it implies that $dim\, {{ \mathbb{ Q} }_\pi }[H] \le dim\, { \mathbb{ Z} } [H]$.  Therefore,  $tr-deg({\mathbb{Q}}(H)/{\mathbb{Q}}) + 1 \le dim\, { \mathbb{ Z} } [H]$. On the other hand, It follows from \cite[Lemma 8]{Sega1977} that $tr-deg({ \mathbb{Q}}(H)/{\mathbb{Q}}) + 1 \ge dim\, {\mathbb{Z}}[H]$.
\end{proof}

Let $J$ be a commutative domain, $G$ be an abelian group of finite rank and let $M$ be a $JG$-module which is $J$-torsion-free. Then it is not difficult to note that there is an integer $n$ for which there exists a free abelian (or trivial) subgroup $X \le G$ of rank $n$ such that the module $M$ is not $JX$-torsion but for any free abelian subgroup $Y \le G$ of rank greater than $n$ the module $M$ is $JY$-torsion. We will say that such an integer $n = tr-deg_{JG}(M)$ is the transcendence degree of the module $M$ over $JG$. 

\begin{theorem}\label{ Theorem 2.4}
 Let $G$ be an abelian minimax group and $M$ be a ${ \mathbb{ Z} } G$-module which has Krull dimension $K(M)$. Then:\par
	(i) if the module $M$ is annihilated by some prime integer $p$ then $K(M) = tr-deg_{{\mathbb{Z}}_p G}(M)$;\par
	(ii) if the module $M$ is ${ \mathbb{ Z} } $-torsion-free then $K(M) = tr-deg_{{\mathbb{Z} } G}(M) + 1$.  
\end{theorem}
\begin{proof}
	(i) Evidently, $M$ can be considered as a ${{ \mathbb{ Z} }_p}G$-module. 
Suppose that $K(M) = n$ then $M$ has a cyclic $n$-critical section $N$. By Lemma 2.2, $G/{C_G}(N) = T \times H$, where $H$ is a finitely generated free abelian group and $T$ is a Chernikov group,  and $N \cong R = {{ \mathbb{ Z} }_p}[T \times H]$, where $R$ is a domain. Therefore, by Proposition 2.3(i), 
$n = tr-deg({\mathbb{Z}}_p (H)/{\mathbb{Z}}_p)$. It easily implies that there is a free abelian subgroup 
$D \le H$ whose generators are algebraically independent over ${ \mathbb{ Z} }_p$ and $r(D) = n$. 
Then there is a free abelian subgroup $X \le G$ such $X \cap {C_G}(N) = 1$ and $X \cdot {C_G}(N)/{C_G}(N) = D$. It easily implies that the section $N$ is not ${{ \mathbb{ Z} }_p}X$-torsion and hence 
$n \le tr-deg_{{\mathbb{ Z}}_p G} (M)$.\par
	Suppose that there is a free abelian subgroup $Y \le G$ such that $M$ is not ${{ \mathbb{ Z} }_p}Y$-torsion and $n < r(Y)$. It easily implies that the module $M$ has a critical ${{ \mathbb{ Z} }_p}Y$-torsion-free section. Therefore,  $r(Y) \le tr-deg({{\mathbb{Z}}_p}(G)/{{ \mathbb{ Z} }_p})$ and it follows from Proposition 2.3(i) that $r(Y) \le K(N) \le K(M) = n$ but it contradicts $n < r(Y)$. Thus,  $n = tr-deg_{{\mathbb{Z}}_p G}(M)$. \par
(ii) If $K(M) = 0$ then $M$ contains a simple ${ \mathbb{ Z} } G$-module $U$. It follows from 
\cite[Theorem 2.3]{Zait1980} that $U$ is annihilated by some prime integer $p$. 
Thus, we can assume that $K(M) \ge 1$.\par
Suppose that $K(M) = n + 1$ then $M$ has a cyclic $n + 1$-critical section $N$. By Lemma 2.2, $G/{C_G}(N) = T \times H$, where $H$ is a finitely generated free abelian group and $T$ is a Chernikov group,  and $N \cong R = { \mathbb{ Z} } [T \times H]$, where $R$ is a domain. Therefore, by Proposition 2.3(ii), $n = tr-deg({ \mathbb{ Q} } (H)/{ \mathbb{ Q} } )$. It easily implies that there is a free abelian subgroup $D \le H$ whose generators are algebraically independent over ${ \mathbb{ Q} } $ and $r(D) = n$. Then there is a free abelian subgroup $X \le G$ such $X \cap {C_G}(N) = 1$ and $X \cdot {C_G}(N)/{C_G}(N) = D$. It easily implies that the section $N$ is not ${ \mathbb{ Z} } X$-torsion and hence $n \le tr-deg_{{\mathbb{Z}} G}(M)$.\par
	Suppose that there is a free abelian subgroup $Y \le G$ such that $M$ is not ${ \mathbb{ Z} } Y$-torsion and $n < r(Y)$. It easily implies that the module $M$ has a critical ${ \mathbb{ Z} } Y$-torsion-free section. Therefore,  $r(Y) \le tr-deg ({\mathbb{Q}}(G)/{\mathbb{Q}})$ and it follows from Proposition 2.3(ii) that $r(Y) + 1 \le K(N) \le K(M) = n + 1$ but it contradicts $n < r(Y)$. Thus, $n = tr-deg_{{\mathbb{Z}}G}(M)$. 
\end{proof}

	If $M$ is an abelian group then $t(M)$ denotes the torsion subgroup of $M$, ${M_p}$ denotes a Sylow subgroup of $t(M)$ for $p \in \pi (t(M))$ and $\Omega ({M_p})$ denotes the subgroup of ${M_p}$ which consists of elements of order $p$.

\begin{corollary}\label{ Corollary 2.5}
 Let $G$ be an abelian minimax group and $M$ be a ${ \mathbb{ Z} } G$-module which has Krull dimension $K(M)$. Then 
  $K(M) = max\{ tr-deg_{{\mathbb{Z}}G}(M/t(M)) + 1,\, \{tr-deg_{{\mathbb{Z}}_p G}(\Omega (M_p))\, |p\,  \in \pi (t(M)) \}\} $.
  \end{corollary}

\section{On metabelian groups with $K(G) \le 1$} \label{Section 3}

\begin{lemma}\label{ Lemma 3.1} Let $G$ be a group with Krull dimension, then $G$ has no $G$-section $B/A =  \times_{i = 1}^\infty ({B_i}/A)$, where ${B_i}$ are normal subgroups of $G$  and ${B_i} \ne A$ for all 
$ i \in \mathbb{N}$.
\end{lemma}
\begin{proof}
	We can use the arguments of the proof of \cite[Lemma 6.2.6]{McRo1988}.
\end{proof}
\begin{lemma}\label{ Lemma 3.2} 
Let $A$ be a torsion-free abelian normal subgroup of a group $G$. 
If $A$ has an infinite strictly ascending chain $\{A_i \, | \, i \in \mathbb{N}\}$ of isolated minimax $G$-invariant 
subgroups $A_i$ then the group $G$ has no Krull dimension. 
\end{lemma}
\begin{proof}
There is no harm in assuming that $ A=\bigcup_{ i \in \mathbb{N}} A_i$. Now, let $i$ be an integer such that $i \ge 0$. Let ${p_i}$ be a prime integer such that ${p_i} \notin Sp({A_i})$. Put ${\pi_0} = 1$ and ${\pi_i} = {\pi_{i - 1}}{p_i}$. Let ${B_i} = \left\langle {{A_j}{\pi_{j - 1}}|j \le i} \right\rangle $, $B = \left\langle {{A_i}{\pi_{i - 1}}|i \ge 1} \right\rangle $,  ${C_i} = \left\langle {{A_j}{\pi_j}|j \le i} \right\rangle $, $C = \left\langle {{A_i}{\pi_i}|i \ge 1} \right\rangle $ and ${D_i} = {C_i}\left\langle {{A_j}{\pi_{j - 1}}|j > i} \right\rangle $. Then ${B_i}{D_i} = B$ and, as the subgroups ${A_i}$ are isolated, we have ${B_i} \cap {D_i} = {C_i}$. It implies that     $B/{C_i} = {B_i}/{C_i} \oplus {D_i}/{C_i}$. Put ${A_0} = 1$ and let ${\bar A_i} = {A_i}/{A_{i - 1}}$. As the subgroups ${A_i}$ are isolated, we see that ${B_i} \cap C = {C_i}$. Then, as $B/{C_i} = {B_i}/{C_i} \oplus {D_i}/{C_i}$, it is not difficult to show that $B/C \cong  \oplus_{i = 1}^\infty ({\bar A_i}{\pi_{i - 1}}/{\bar A_i}{\pi_i})$ and the assertion follows from Lemma 3.1. 
\end{proof}

\begin{proposition}\label{ Proposition 3.3}
 Let $G$ be a metabelian group with an abelian normal subgroup $A \le G$ such that the quotient group $B = G/A$ is finitely generated abelian. Suppose that the quotient group $G/t(A)$ has Krull dimension and has no ${ \mathbb{ Z} } \wr { \mathbb{ Z} } $-sections. Then the quotient group $G/t(A)$ is minimax and  $K(G/t(A)) \le 1$.
\end{proposition}
\begin{proof}
	Passing to the quotient group $G/t(A)$  we can assume that the subgroup $ A $ is torsion-free.  Let $C$ be a finitely generated subgroup of $G$ such that $G = AC$, let $x \in G$ and let $X = \left\langle {x,\,C} \right\rangle $. Then $X$ is a finitely generated metabelian group and, by \cite{Hall1959}, $X$ meet the maximal condition for normal subgroups. Therefore, $X$ has Krull dimension. As $G$ has no ${\mathbb{Z}} \wr {\mathbb{Z}} $-sections, we can conclude that $X$ has no ${ \mathbb{ Z} }\wr {\mathbb{Z}} $-sections. Then it follows from \cite[Corollary B1]{JaKr2020} that $X$ is of finite rank and, and it follows from by \cite[Lemma 6]{Hall1959}, $X$ is minimax. Therefore, $A \cap X = {X_1}$ is a $G$-invariant minimax subgroup of $A$. It follows from Lemma 3.1 that $\left| {\pi (A/{X_1})} \right| < \infty $ and hence ${A_1} = i{s_A}({X_1})$ is a $G$-invariant isolated minimax subgroup of $A$. Applying the above arguments to the quotient group $G/{A_1}$ we can construct an ascending chain of minimax $G$-invariant isolated minimax subgroups of $A$. It follows from Lemma 3.2 that this chain is finite and hence the subgroup $A$ is minimax. 
\end{proof}

	It immediately follows from the above proposition that the group constructed in \cite[Example 2.13]{JaKr2020} has no Krull dimension. 

\begin{theorem}\label{Theorem 3.4}
 Let $G$ be a metabelian group with an abelian torsion-free normal subgroup $M \le G$ such that the quotient group $B = G/M$ is abelian. Suppose that the group $G$ has Krull dimension $K(G)$. Then $K(G) \le 1$ if and only if $G$ has no ${ \mathbb{ Z} } \wr { \mathbb{ Z} } $-sections. 
\end{theorem}
\begin{proof}
 By \cite[Lemma 4.4]{Tush2003}, the group $B$ is abelian minimax, We can consider $M$ as a ${ \mathbb{ Z} } $-torsion-free ${ \mathbb{ Z} } B$-module. If $K(M) = 0$ then $M$ contains a simple ${ \mathbb{ Z} } G$-module $U$. It follows from \cite[Theorem 2.3]{Zait1980} that $U$ is annihilated by some prime integer $p$. Thus, we can assume that $K(M) = 1$. \par
It follows from Theorem 2.4 that $K(M) = tr-deg_{{\mathbb{Z}}B}(M) + 1$ and hence $K(M) = 1$ if and only if $tr-deg_{{\mathbb{Z}}B}(M) = 0$. The last condition means that for any element of infinite order $g \in B$ the module $M$ is ${ \mathbb{ Z} } \left\langle g \right\rangle $-torsion, i.e. $M$ has no ${ \mathbb{ Z} } \wr { \mathbb{ Z} } $-sections. 
\end{proof}

\begin{corollary}\label{Corollary 3.5}
 Let $G$ be a metabelian group with an abelian normal subgroup $A \le G$ such that the quotient group $B = G/A$ is finitely generated abelian. Suppose that the quotient group $G/t(A)$ has Krull dimension $K(G) \le 1$  Then the quotient group $G/t(A)$ is minimax.
\end{corollary}
\begin{proof}
The Corollary follows from Proposition 3.3 and the above theorem. 
\end{proof}

\begin{proposition}\label{ Proposition 3.6}
 Let $G$ be a metabelian group with an abelian normal subgroup $A \le G$ such that the quotient group $B = G/A$ is finitely generated abelian. Suppose that the group $G$ has Krull dimension $K(G) \le 1$. Then the group $G$ satisfies the condition $Min - \infty  - N$. 
\end{proposition}
\begin{proof} 
We consider $A$ as a ${ \mathbb{ Z} } B$-module. By \cite[Theorem 4.3]{Tush2003}, $A$ has a finite series of submodules each of whose factor is either Artenian or Noetherian. Then it is sufficient to show that any Noetherian 
${\mathbb{Z}}B$-module $M$ with $K(M) \le 1$ satisfies the weak condition 
of minimality  for submodules. Evidently the module $M$ has a finite series of cyclic submodules each of whose factor is either simple or 1-critical. Then it is sufficient to show that any cyclic 1-critical ${ \mathbb{ Z} } B$-module $N$  satisfies the weak condition of minimality for submodules. Since  the module $N$ is 1-critical, for any nonzero submodule $X \le N$ the quotient module $N/X$ is cyclic Artenian. By \cite[Theorem 4]{Hall1959} the group ring ${ \mathbb{ Z} } B$ is Noetherian and hence the quotient module $N/X$ is Noetherian. It follows from \cite[Theorem 3.1]{Hall1959} that any simple ${ \mathbb{ Z} } B$-module is finite and it implies that the quotient module $N/X$ is finite. Thus, the assertion follows.
\end{proof}
\begin{example}\label{ Example 3.7}
 There exists a torsion-free metabelian group $G$ with Krull dimension $K(G) = 1$ which is not minimax and does not satisfy the condition $Min - \infty  - N$. 
\end{example}
\begin{proof}
Let $F$ be an algebraically closed field of characteristic zero, $H$ be a quasicyclic  $p$-subgroup of ${F^*}$ and $g = \frac{1}{p}$. Let $R = { \mathbb{ Z} } [H \times \left\langle g \right\rangle ]$ then, by \cite [Lemma 3.2]{Tush2003}, $K(R) = 1$. The additive group ${R^ + } \cong  \oplus_{i = 1}^\infty {({{ \mathbb{ Q} }_p})_i}$ Therefore, ${R^ + }$ is not minimax and for any prime integer $q \ne p$ we have infinite descending chain 
$\{ R{q^n}\, |\, n \in {\mathbb{Z} }  \}$ of ideals with infinite sections. Then the split extension 
$G = R\leftthreetimes (H \times \left\langle g \right\rangle )$ has Krull dimension $K(G) = 1$, is not minimax and does not satisfy the condition $Min- \infty - N$. 
\end{proof}

The author is deeply grateful to Peter Krafuller for personal communication which stimulated the appearance of this article.


\begin{thebibliography}{99}%

\bibitem{GKR1} F. De Giovanni, L.A. Kurdachenko, A. Russo, Groups with pronormal deviation, J. Algebra 613 (2023) 32-45.
\bibitem{GKR2} F. De Giovanni, L.A. Kurdachenko, A. Russo, Groups with subnormal deviation, Mathematics 11 (2023) article n. 2635. 
\bibitem{DaMa2024}  U. Dardano, F. De Mari, A real chain condition for groups, J. Algebra 642 (2024) 451-469. 
\bibitem{Hall1959} Ph. Hall, Finiteness conditions for soluble groups, Proc. London Math. Soc. 4 (1954) 419-436.
\bibitem{HaMc1971} B. Hartley, D. McDougall, Injective modules and soluble groups satisfying the minimal condition for normal subgroups, Bull. Austral. Math. Soc. 4 (1971) 113-135.
\bibitem{JaKr2020} L. Jacoboni, P. Kropholler, Soluble groups with no  
${\mathbb{Z}} \wr {\mathbb{Z}}$-sections, Annales Henri Lebesgue 3 (2020) 981-998.
\bibitem{KuSm2009} L.A. Kurdachenko, H. Smith, Groups with small deviation for non-subnormal subgroups, Centr. Eur. J. Math. 7 (2009) 186-199.
\bibitem{McRo1988} J.C. McConnell, J.C. Robson, Noncommutative Noetherian rings, (John Wiley and sons. 1988). 
\bibitem{Sega1977} D. Segal, On the residual simplicity of certain modules, Proc. London Math. Soc. (3) 34 (1977) 327-353. 
\bibitem{Tush2002} A.V. Tushev, On Noetherian modules over minimax abelian groups, Ukr. Math. J. (7) 54  (2002) 1169-1180.
\bibitem{Tush2003} A. V. Tushev, On deviation in groups,  Illinois J. Math. 47 (2003) 539-550. 
\bibitem{Zait1980} D. I. Zaitsev, Products of abelian groups, Algebra Logika (2) 19 (1980) 150-172. 
\bibitem{ZaKT1985} D. I. Zaitsev, L. A. Kurdachenko, A. V. Tushev, Modules over nilpotent groups of finite rank, Algebra and Logika (6) 24 (1985) 631-666.

\end{thebibliography}
\end{document}